\documentclass[a4paper,10pt,reqno]{amsart}

\usepackage[utf8]{inputenc}
\usepackage[T1]{fontenc}
\usepackage{amsthm}
\usepackage{amsmath}
\usepackage{amssymb}
\usepackage[inline]{enumitem}
\usepackage{comment}
\PassOptionsToPackage{hyphens}{url}\usepackage{hyperref}
\usepackage{fancyhdr}
\usepackage{mathrsfs}
\usepackage{stmaryrd}
\usepackage[normalem]{ulem}
\usepackage{xcolor}
\usepackage{tikz-cd}
\usetikzlibrary{arrows}

\usepackage{cite}
\setlist[description]{%
  itemsep=0.05cm,               
  font={\normalfont\textsc}, 
 leftmargin=\parindent,
 labelindent=\parindent
}

\theoremstyle{definition}

\newtheorem{theorem}{Theorem}[section]
\newtheorem{lemma}[theorem]{Lemma}
\newtheorem{proposition}[theorem]{Proposition}
\newtheorem{corollary}[theorem]{Corollary}

\theoremstyle{definition}
\newtheorem{definition}[theorem]{Definition}

\newtheorem{remark}[theorem]{Remark}

\newtheorem*{conjecture*}{Conjecture}
\newtheorem*{problem*}{Problem}
\newtheorem*{question*}{Question}

\definecolor{blue-url}{RGB}{0,0,100}
\definecolor{red-url}{RGB}{100,0,0}
\definecolor{green-url}{RGB}{0,100,0}
\definecolor{light-yellow}{RGB}{255,255,128}
\definecolor{light-blue}{RGB}{193,255,255}
\definecolor{light-red}{RGB}{239,83,80}

\hypersetup{
	pdftitle={On the automorphisms of numerical power monoids},
	pdfauthor={Bhowmik and Tringali},
	pdfmenubar=false,
	pdffitwindow=true,
	pdfstartview=FitH,
	colorlinks=true,
	linkcolor=blue-url,
	citecolor=green-url,
	urlcolor=red-url
}

\renewcommand{\emptyset}{\varnothing}
\renewcommand{\setminus}{\smallsetminus}
\renewcommand{\,}{\kern 0.1em}

\DeclareMathOperator{\rev}{rev}

\providecommand\llb{\llbracket}
\providecommand\rrb{\rrbracket}

\newcommand{\evid}[1]{\textsf{#1}}
\newcommand{\fin}{\mathrm{fin}}
\newcommand{\aut}{\mathrm{Aut}}

{\newline\vspace{\abovedisplayskip}\hbox to \textwidth\bgroup\hss$\displaystyle}
{$\hss\egroup\vspace{\belowdisplayskip}}

\makeatletter
\DeclareFontFamily{OMX}{MnSymbolE}{}
\DeclareSymbolFont{MnLargeSymbols}{OMX}{MnSymbolE}{m}{n}
\SetSymbolFont{MnLargeSymbols}{bold}{OMX}{MnSymbolE}{b}{n}
\DeclareFontShape{OMX}{MnSymbolE}{m}{n}{
	<-6>  MnSymbolE5
	<6-7>  MnSymbolE6
	<7-8>  MnSymbolE7
	<8-9>  MnSymbolE8
	<9-10> MnSymbolE9
	<10-12> MnSymbolE10
	<12->   MnSymbolE12
}{}
\DeclareFontShape{OMX}{MnSymbolE}{b}{n}{
	<-6>  MnSymbolE-Bold5
	<6-7>  MnSymbolE-Bold6
	<7-8>  MnSymbolE-Bold7
	<8-9>  MnSymbolE-Bold8
	<9-10> MnSymbolE-Bold9
	<10-12> MnSymbolE-Bold10
	<12->   MnSymbolE-Bold12
}{}

\let\llangle\@undefined
\let\rrangle\@undefined
\DeclareMathDelimiter{\llangle}{\mathopen}%
{MnLargeSymbols}{'164}{MnLargeSymbols}{'164}
\DeclareMathDelimiter{\rrangle}{\mathclose}%
{MnLargeSymbols}{'171}{MnLargeSymbols}{'171}
\makeatother

\begin{document}
\title{On the automorphisms of numerical power monoids}

\author{Anwita Bhowmik}
\address{(A.~Bhowmik) School of Mathematical Sciences, Hebei Normal University | Shijiazhuang, Hebei province, 050024 China}
\email{bhowmikanwita@gmail.com}
\urladdr{https://anwitab.github.io/anwita\_website/}

\author{Salvatore Tringali}
\address{(S.~Tringali) School of Mathematical Sciences, Hebei Normal University | Shijiazhuang, Hebei province, 050024 China}
\email{salvo.tringali@gmail.com}
\urladdr{https://salvo-tringali.github.io/home/}

%


\begin{abstract}
Let $H$ be a numerical monoid, that is, a cofinite submonoid of $\mathbb N$ (the non-negative integers under addition). Denote by $\mathcal P_{\text{fin},0}(H)$ the monoid obtained by endowing the family of all finite subsets of $H$ containing $0$ with the operation of setwise addition induced by $H$ on its power set.

Tringali and Yan [JCTA, 2025] have recently established that $\mathcal P_{\text{fin},0}(\mathbb N)$ has a unique non-trivial automorphism, and conjectured that the automorphism group of $\mathcal P_{\fin,0}(H)$ is trivial whenever $H \ne \mathbb N$. We prove this conjecture and, as a byproduct, give a new proof of the Tringali--Yan theorem.
\end{abstract}

\maketitle

\thispagestyle{empty}

\section{Introduction}
\label{sec:intro}

Let $H$ be a semigroup, written multiplicatively (see Section \ref{subsec:notation-terminology} for notation and terminology). The family of all non-empty
subsets of $H$ is then itself a semigroup under the binary operation 
\[
(X, Y) \mapsto XY:=\{xy:x\in X,\, y\in Y\}.
\]
This semigroup is denoted by $\mathcal P(H)$ and called the
\evid{large power semigroup} of $H$; other names used in the earlier
literature include \evid{global} and
\evid{semigroup of complexes}. The subsemigroup
$\mathcal P_\fin(H)$ of $\mathcal P(H)$ consisting of the non-empty finite subsets of $H$ is known as the \evid{finitary power semigroup} of $H$.

If, in particular, $H$ is a monoid with identity $1_H$, then both
$\mathcal P(H)$ and $\mathcal P_\fin(H)$ are monoids as well, their identity being the singleton $\{1_H\}$. Moreover, the families
\[
\mathcal P_1(H):=\{X\in\mathcal P(H):1_H\in X\}
\quad\text{and}\quad
\mathcal P_{\fin,1}(H):=\mathcal P_1(H)\cap\mathcal P_\fin(H)
\]
are submonoids of $\mathcal P(H)$ and $\mathcal P_\fin(H)$,
respectively. The former is dubbed the \evid{reduced large power
monoid} of $H$, and the latter the \evid{reduced finitary power monoid}
of $H$.

\begin{remark}
Depending on the context, we refer generically to any of the above structures
as a \evid{power semigroup} or a \evid{power monoid}. Note also that,
when $H$ is written additively, we adopt the same notational convention for every
subsemigroup of $\mathcal P(H)$. In practice, this translates into the fact that the
operation on $\mathcal P(H)$ maps a pair $(X,Y)$ of non-empty subsets
of $H$ to their \evid{sumset}
\[
X+Y:=\{x+y:x\in X,\ y\in Y\},
\]
as opposed to the \evid{setwise product} of $X$ and $Y$ in the multiplicative setting. By the same token, we use $\mathcal P_0(H)$ and $\mathcal P_{\fin,0}(H)$ instead of $\mathcal P_1(H)$ and $\mathcal P_{\fin,1}(H)$, respectively.
\end{remark}

Power semigroups were first systematically studied by Tamura and Shafer \cite{Tam-Sha1967}
in the late 1960s, though their explicit appearance in the literature can be traced back at least to Ballieu's 1950 paper \cite{Bal1950}. Among other things, power semigroups serve as an intrinsic algebraic framework for many
questions lying at the intersection of arithmetic
combinatorics and (semi)group theory  \cite{Bie-Ger-22, Tri-Yan-23(a), Tri-Yan2025(b), GarSan-Tri-24(a), Tri-Yan2026(a), Tri-Wen-2026(b), Tri-Wen-2026(a)}, while providing a flexible
benchmark for the study of (non-unique) factorizations beyond the
classical commutative and cancellative setting \cite{Fa-Tr18, An-Tr18, Tr20(c), Co-Tr-25(a), Rei-2026}.
For further details on these connections, see the survey \cite{Tri-Survey}.

A basic question dating back to Tamura and Shafer's work asks to what extent the structure of a semigroup can
be recovered from its large power semigroup, and a number of variants involving other power constructions have been considered over the years \cite[Section 1]{GarSan-Tri-24(a)}. The starting observation is that any semigroup isomorphism $f \colon H \to K$ induces an isomorphism $f^\ast \colon \mathcal P(H) \to \mathcal P(K)$ via $X \mapsto \allowbreak f[X]$ (the direct image of $X$ under $f$).
Since $f[X]$ is finite for every finite $X \subseteq H$, the map $f^\ast$ restricts to an isomorphism $\mathcal P_\fin(H) \to \mathcal P_\fin(K)$. Moreover, if $H$ is a monoid, then so is $K$, and $f$ sends $1_H$ to $1_K$ (see, e.g., \cite[Lemma 1.1]{Gou-Isk-Tsi-1984} and the last lines of \cite[Section 2]{Tri-2024(a)}); accordingly, $f^\ast$ also restricts to iso\-mor\-phisms $
\mathcal P_1(H) \to \mathcal P_1(K)$ and $
\mathcal P_{\fin,1}(H) \to \allowbreak \mathcal P_{\fin,1}(K)$.
An isomorphism between power semigroups or power monoids is said to be \evid{inner} if it arises in this way, and it is called \evid{outer} otherwise.

We are thus led to the problem of determining the outer automorphisms, if any, of a power semigroup. The reduced finitary power monoid $\mathcal P_{\fin,0}(\mathbb N)$ of the additive monoid $\mathbb N$ of non-negative integers provides a particularly instructive example. Indeed, it is easy to verify that the identity is the only automorphism of $\mathbb N$. Thus, every non-trivial automorphism of $\mathcal P_{\fin,0}(\mathbb N)$ is outer, and Bienvenu and Geroldinger observed in \cite[p.~13]{Bie-Ger-22} that at least one such automorphism exists, namely the \evid{reversion map}
\begin{equation}
\label{equ:reversion-map}
\rev \colon
X \mapsto \max(X)-X,
\qquad\text{where }
{-X} := \{-x: x \in X\} \subseteq \mathbb Z.
\end{equation} 
Tringali and Yan subsequently showed in \cite[Theorem~3.2]{Tri-Yan2025(b)} that there are no others, unaware that the same result had been obtained by Byrd et al.~in~\cite[Theorem~1]{Byr-Llo-Ped-Ste-1984}. They also conjectured in \cite[Section~4]{Tri-Yan2025(b)} that $\aut(\mathcal P_{\fin,0}(H))$ is trivial for every numerical monoid $H \ne \mathbb N$, where by a \evid{numerical monoid} we mean a cofinite (additive) submonoid of $\mathbb N$.
The principal goal of the present paper is to settle this conjecture in the affirmative. Along the way, we also give a new and more conceptual proof that $\aut(\mathcal P_{\fin,0}(\mathbb N))$ is cyclic of order two. That is, we obtain the following theorem, whose proof is deferred to Section \ref{sec:proof-of-main-thm}.

\begin{theorem}
\label{thm:main}
If $H$ is a numerical monoid, then $\aut(\mathcal P_{\fin,0}(H))$ is trivial unless $H = \mathbb N$, in which case the only non-trivial automorphism is the reversion map.
\end{theorem}

A couple of new ingredients used in the argument are the notion of ``autocorrelation'' introduced in Definition~\ref{def:h-adic-valuation} and the invariance property established in Proposition~\ref{prop:h-valuation-is-preserved}, both of which may be of independent interest. The latter relies, in turn, on Tringali and Yan's two-to-two lemma~\cite[Corollary~3.3]{Tri-Yan2026(a)} and Rago's cardinal lemma~\cite[Lemma~4]{Rago26(c)} (we slightly generalize this second result in Proposition~\ref{prop:rago-cardinal-lemma}).

Turning back to the literature, Wong et al.~proved in \cite[Theorem~1.3]{Wong2025} that the automorphism group of the finitary power semigroup $\mathcal P_{\fin}(S)$ of a numerical semigroup $S$ is trivial unless $S = \llbracket k, \infty \rrbracket$ for some integer $k \geq 0$, in which case $\aut(\mathcal P_{\fin}(S))$ is cyclic of order two. Here, a \evid{numerical semigroup} is a cofinite subsemigroup of $\mathbb N$, so a numerical monoid is precisely a numerical semigroup containing zero. 

As a follow-up, Tringali and Wen proved in \cite[Theorem~3.5]{Tri-Wen-2026(a)} that the automorphism group of the finitary power monoid $\mathcal P_{\fin}(\mathbb Z)$ of the additive group $\mathbb Z$ of integers is isomorphic to the direct product of a cyclic group of order two and the infinite dihedral group, in stark contrast to the fact that the only non-trivial automorphism of $\mathbb Z$ itself is the map $x \mapsto -x$. They also conjectured \cite[Section~4]{Tri-Wen-2026(a)} that $\aut(\mathcal P_{\fin,0}(\mathbb Z))$ is cyclic of order two, which was confirmed by Wong et al.~in \cite{Wong2026}.

In fact, Byrd et al.~had already established in~\cite[Theorem~2]{Byr-Llo-Ped-Ste-1984} that the automorphism group of $\mathcal P_\fin(\mathbb Z)$ is a splitting extension of $\mathbb Z$ by the Klein four-group $V_4$; that is, there exists an action $\phi$ of $V_4$ on $\mathbb Z$ such that
$\aut(\mathcal P_\fin(\mathbb Z))$ is isomorphic to the semidirect product $\mathbb Z \rtimes_\phi V_4$. However, since $V_4$ admits both trivial and non-trivial actions on $\mathbb Z$, this description does not by itself determine the isomorphism type of $\aut(\mathcal P_\fin(\mathbb Z))$. More generally, an analogous description of $\aut(\mathcal P_\fin(G))$ for every non-trivial additive subgroup $G$ of the rationals was obtained in \cite[Corollary~4.10]{Byr-Llo-Ste-1982}.

Rago \cite[Theorem~13]{Rago26(b)} then proved that, for a
finite abelian group $G$, the automorphism group of
$\mathcal P_{\fin,1}(G)$ is canonically isomorphic to $\aut(G)$,
except when $G \cong V_4$, in which case
$\aut(\mathcal P_{\fin,1}(G))$ is isomorphic to the direct product of
two copies of the symmetric group of degree three
\cite[Example~2]{Rago26(b)}. 
On a related note, results of Byrd
et al.~\cite[Theorems~2, 3, and~5]{Byr-Llo-Ped-Ste-1977} show that $\aut(\mathcal P_\fin(G))$ is canonically isomorphic to $\aut(G)$ whenever $G$ is either a finite cyclic group of order $n \notin \{3,4,5\}$ or a subgroup of the circle group of order greater than $5$. The same paper also describes $\aut(\mathcal P_\fin(G))$ when $G$ is cyclic of order $3$, $4$, or $5$ (ibid., Section~3, p.~31). Here, ``canonically'' means that augmentation induces a group isomorphism from $\aut(G)$ onto either $\aut(\mathcal P_{\fin,1}(G))$ or $\aut(\mathcal P(G))$, depending on the context.

As for future directions, it would be interesting to extend Theorem \ref{thm:main} from
numerical monoids to rational monoids (that is, submonoids of the rational numbers under addition) or to affine monoids
(that is, finitely generated submonoids of the direct sum of finitely
many copies of $\mathbb Z$).

\section{Preliminaries}
\label{sec:preliminaries}

Below, we first fix the notation and terminology used throughout the paper, and then collect some general preliminary results on power semigroups that will come in handy later.

\subsection{Generalities.}
\label{subsec:notation-terminology}
We use $\mathbb N$ for the additive monoid of non-negative integers, $\mathbb N^+$ for the set of positive integers, and $\mathbb Z$ for additive group of integers.
We write $|S|$ for the cardinality of a set $S$, and given a function $f \colon A \to B$, we let $f[X] := \allowbreak \{f(x): \allowbreak x \in X\}$ for every $X \subseteq A$.

Unless stated otherwise, we reserve the letters $i$ and $j$ (with or without subscripts) for non-negative integers, and the letters $m$ and $n$ for positive integers.
If $a, b \in \mathbb Z \cup \{\pm \infty\}$, then
\[
\llb a, b \rrb := \{k \in \allowbreak \mathbb Z \colon \allowbreak a \le k \le b\}
\]
is the (\evid{discrete}) \evid{interval} from $a$ to $b$.
If $k \in \mathbb N$, then $kX$ is the $k$-fold sum of a subset $X$ of $\mathbb Z$. Note that $\llb a, b \rrb = \emptyset$ when $b < a$; moreover, $0X = \{0\}$. 

We refer to Howie's monograph \cite{Ho95} for the basics of semigroup theory.
In particular, all morphisms are understood to be semigroup homomorphisms and, unless otherwise specified, all semigroups are written multiplicatively.
Further notation and terminology, if not explained when first used, are standard or should be clear from the context.

\subsection{Preparations}
We start with a result of Tringali and Yan~\cite[Corollary~3.3]{Tri-Yan2026(a)}, which plays a fundamental role in the study of isomorphism problems for reduced finitary power monoids:

\begin{proposition}
\label{prop:two-to-two-lemma}
Let $H$ and $K$ be arbitrary monoids, and let $f$ be an isomorphism from $\mathcal P_{\fin,1}(H)$ to $\mathcal P_{\fin,1}(K)$. There exists a bijection $g \colon H \to K$ such that $f(\{1_H, x\}) = \{1_K, g(x)\}$ for every $x \in H$.
\end{proposition}

The bijection $g$ in Proposition~\ref{prop:two-to-two-lemma} is called the \evid{pullback} of the isomorphism $f$. However, in general, $g$ need not be an isomorphism, even in the special case where $H$ is an additive submonoid of a
free abelian group of rank two and $K$ coincides with $H$~\cite[Corollary~3]{Rago26}. 

We continue with a remarkable result of
Rago~\cite[Lemma~4]{Rago26(c)} concerning an invariance property of isomorphisms between the reduced finitary power monoids of
cancellative commutative monoids $H$ and $K$, both containing an element of infinite order. This is in fact the only version we shall use later. Here, however, we extend the result to arbitrary cancellative monoids and only assume that $H$ contains an element of
infinite order (Proposition~\ref{prop:rago-cardinal-lemma}). The proof closely follows Rago's original argument, although we present it in a more modular form and start with a pair of lemmas.

Given a semigroup $H$, we recall that an element $a \in H$ is \evid{left-cancellative} if left multiplication by $a$ is injective on $H$. Right cancellativity is defined in a symmetric way. An element is \evid{cancellative} if it is both left- and right-cancellative, and the semigroup itself is cancellative if each of its elements is.

\begin{lemma}
\label{lem:aperiodicity-implies-disjointness}
Let $H$ be a cancellative monoid, and assume that $a \in H$ is an element of infinite order. Given a set $X \in \mathcal P_{\fin,1}(H)$, there exists an integer $N \ge 1$ (depending on $X$) such that
\begin{equation}
\label{lem:aperiodicity-implies-disjointness:eq(1)}
a^{in} X \cap a^{jn} X = \emptyset,
\qquad \text{for all } n \ge N \text{ and } i, j \in \mathbb N \text{ with } i \ne j.
\end{equation}
\end{lemma}

\begin{proof}
Suppose by way of contradiction that $X \cap a^n X$ is non-empty for infinitely many $n \in \mathbb N^+$. Since $X$ is finite, the pigeonhole principle yields distinct $r, s \in \mathbb N^+$ and $x, y \in X$ such that $a^r x = y = a^s x$.
By right cancellativity, it follows that $a^r = a^s$, contradicting the assumption that $a$ has infinite order in $H$. Consequently, there exists an integer $N_X \ge 1$ such that
\begin{equation}
\label{lem:aperiodicity-implies-disjointness:eq(2)}
X \cap a^n X = \emptyset,
\qquad\text{for every }n \ge N_X. 
\end{equation}
We claim that Eq.~\eqref{lem:aperiodicity-implies-disjointness:eq(1)} holds with $N = N_X$. Otherwise, $a^{in}u = a^{jn}v$ for some $n \ge N$, distinct $i, j \in \mathbb N$, and $u, v \in \allowbreak X$. Without loss of generality, we may assume $i < j$. Then, by left cancellativity, $u = a^{(j-i)n}v$,
and hence $X \cap a^{(j-i)n}X \ne \emptyset$.
However, since $(j-i)\,n \ge n \ge N_X$, this contradicts
Eq.~\eqref{lem:aperiodicity-implies-disjointness:eq(2)}.
\end{proof}

\begin{lemma}
\label{lem:order-is-preserved}
Let $H$ and $K$ be monoids, and let $g$ be the pullback of an isomorphism $f \colon \mathcal P_{\fin,1}(H) \to \mathcal P_{\fin,1}(K)$. If $a$ is an element of infinite order in $H$, then $g(a)$ has infinite order in $K$. Moreover, 
\begin{equation}
\label{lem:order-is-preserved:eq(1)}
g(a^m) = g(a)^m,
\qquad\text{for all }
m \in \mathbb N.
\end{equation}
\end{lemma}

\begin{proof}
Set $b := g(a)$. By
\cite[Proposition~4.1]{Tri-Yan2026(a)}, $b$ has infinite
order in $K$. So, we focus on proving that
\[
g(a^m) = b^m,
\qquad\text{for all }
m \in \mathbb N.
\]

The claim is trivial for $m = 0$, because $f(\{1_H\}) = \{1_K\}$ and hence $g(1_H) = 1_K$. If $m \geq 1$, then
\[
\{1_H, a^m\}\{1_H, a\}^{m-1} = \{1_H, a^m\} \{1_H, a, \ldots, a^{m-1}\} = \{1_H, a, \ldots, a^{2m-1}\} = \{1_H, a\}^{2m-1}.
\]
Applying $f$ and recalling the definition of the pullback, it follows that
\[
g(a^m) \in \{1_K, g(a^m)\}\{1_K, b\}^{m-1} = f(\{1_H, a^m\}) f(\{1_H, a\})^{m-1} = f(\{1_H, a\})^{2m-1} = \{1_K, b\}^{2m-1}.
\]
In particular, this yields $g(a^m) = b^\ell$ for some
$\ell\in\llbracket 0,2m-1\rrbracket$, and therefore
\begin{equation*}
\{1_K, b^\ell\} \{1_K,b\}^{m-1}
   =\{1_K,b\}^{2m-1}.
\end{equation*}
Since $b$ has infinite order, its non-negative powers are pairwise distinct. Thus, comparing the exponents on the two sides of the last display gives $\ell = m$, as desired.
\end{proof}

Incidentally, Lemma~\ref{lem:order-is-preserved} may be regarded as a complement to \cite[Proposition~4.3]{Tri-Yan2026(a)}, where Eq.~\eqref{lem:order-is-preserved:eq(1)} is established under the assumption that both $H$ and $K$ are cancellative monoids, irrespective of whether $a$ has infinite order in $H$.

\begin{proposition}
\label{prop:rago-cardinal-lemma}
Let $H$ and $K$ be cancellative monoids, and assume that $H$ contains at least one element of infinite order. Every isomorphism $f \colon \mathcal P_{\fin,1}(H) \to \mathcal P_{\fin,1}(K)$ is \evid{cardinality-preserving}, that is,
\[
|A| = |f(A)|, \qquad\text{for all } A \in \mathcal P_{\fin,1}(H).
\]
\end{proposition}

\begin{proof}
Let $g \colon H \to K$ be the pullback of $f$. By hypothesis, $H$
contains an element of infinite order, say $a$. By
Lemma~\ref{lem:order-is-preserved}, the image $b := g(a)$ has infinite order in $K$. In addition, we have
\begin{equation}
\label{prop:rago-cardinal-lemma:eq(0)}
f(\{1_H, a^m\}) = \{1_K, b^m\},
\qquad\text{for every }
m \in \mathbb N.
\end{equation}

Fix $A \in \mathcal P_{\fin,1}(H)$, and set $B := f(A)$. By Lemma \ref{lem:aperiodicity-implies-disjointness} (applied first to the pair $(a, A)$ and then to the pair $(b, B)$), there exists an integer $n \ge 1$ such that
\begin{equation}
\label{prop:rago-cardinal-lemma:eq(1)}
a^{in} A \cap a^{jn} A = b^{in} B \cap b^{jn} B = \emptyset,
\qquad \text{for all } i, j \in \mathbb N \text{ with } i \ne j.
\end{equation}
Accordingly, let $\phi$ be the mapping
\[
2^H \to 2^H \colon V \mapsto A \cup a^{2n} A \cup V,
\]
where $2^H$ is the power set of $H$.
We claim that the solutions $X \in \mathcal P_{\fin,1}(H)$ to the equation
\begin{equation}
\label{lem:rago-cardinal-lemma:eq(4)}
\{1_H, a^n\}X = \{1_H, a^n\}^3 A
\end{equation}
are precisely the sets $\phi(V)$ with $V \subseteq a^n A$, and hence
the equation has exactly $2^{|A|}$ solutions.

The second assertion is straightforward from the first. In fact, left
cancellativity gives $|a^n A| = |A|$, implying that $a^nA$ has $2^{|A|}$ subsets. Moreover, it is clear from
Eq.~\eqref{prop:rago-cardinal-lemma:eq(1)} that the restriction of
$\phi$ to the power set of $a^n A$ is injective. We therefore focus below on
the first part of the claim.

Note first that if $V \subseteq a^n A$, then $\phi(V) \in \mathcal P_{\fin,1}(H)$. Moreover, $\{1_H, a^n\}V \subseteq a^n A \cup a^{2n}A$, and hence
\[
\begin{split}
\{1_H, a^n\} (A \cup a^{2n} A) & \subseteq \{1_H, a^n\} \phi(V) = \{1_H, a^n\} (A \cup a^{2n} A) \cup \{1_H, a^n\} V \\
& \subseteq \{1_H, a^n\} (A \cup a^{2n} A) \cup (a^n A \cup a^{2n}A) = \{1_H, a^n\} (A \cup a^{2n} A).
\end{split}
\]
So, $\phi(V)$ is indeed a solution of Eq.~\eqref{lem:rago-cardinal-lemma:eq(4)}, upon observing that 
\[
\{1_H, a^n\} (A \cup a^{2n} A) = \{1_H, a^n\}^3 A.
\]

Conversely, suppose that
$X \in\mathcal P_{\fin,1}(H)$ is a solution to Eq.~\eqref{lem:rago-cardinal-lemma:eq(4)}, so that
\begin{equation}
\label{prop:rago-cardinal-lemma:inclusion}
X \cup a^n X = A \cup a^n A \cup a^{2n} A \cup a^{3n} A.
\end{equation}
It is readily checked that $X \cap a^{3n}A = \emptyset$. Otherwise, Eq.~\eqref{prop:rago-cardinal-lemma:inclusion} would imply
\[
\emptyset \ne a^n X \cap a^{4n} A \subseteq (A \cup a^n A \cup a^{2n} A \cup a^{3n} A) \cap a^{4n} A = \bigcup_{i = 0}^3 (a^{in} A \cap a^{4n}A),
\]
contrary to Eq.~\eqref{prop:rago-cardinal-lemma:eq(1)}. 
Consequently, again from Eq.~\eqref{prop:rago-cardinal-lemma:inclusion}, we get that
\begin{equation}
\label{prop:rago-cardinal-lemma:inclusion(2)}
X \subseteq A \cup a^n A \cup a^{2n} A,
\end{equation}
and hence
\begin{equation}
\label{prop:rago-cardinal-lemma:inclusion(3)}
a^n X \subseteq a^n A \cup a^{2n} A \cup a^{3n} A.
\end{equation}
By Eq.~\eqref{prop:rago-cardinal-lemma:inclusion(2)}, the elements of $A$ on the right-hand side of 
Eq.~\eqref{prop:rago-cardinal-lemma:inclusion} can only be contributed by $X$, and so
$A \subseteq X$. Likewise, by Eq.~\eqref{prop:rago-cardinal-lemma:inclusion(3)}, the elements of $a^{3n}A$ can only be contributed by $a^n X$. Thus $a^{3n} A \subseteq a^n X$, and hence $a^{2n} A \subseteq X$ (by cancellativity). Combining these inclusions, we conclude that
\[
A \cup a^{2n} A \subseteq X
\subseteq A\cup a^nA\cup a^{2n}A.
\]
It follows that $X = \phi(V)$ with $V := X \cap a^n A \subseteq a^n A$, which completes the proof of the claim.

Now, using that $f$ is an isomorphism,
we gather from Eqs.~\eqref{prop:rago-cardinal-lemma:eq(0)} and~\eqref{lem:rago-cardinal-lemma:eq(4)} that also the equation
\begin{equation}
\label{prop:rago-cardinal-lemma:companion-equ}
\{1_K, b^n\} Y
   = \{1_K,b^n\}^3 B
\end{equation}
has precisely $2^{|A|}$ solutions $Y \in \mathcal P_{\fin,1}(K)$. On the other hand, since $K$ is cancellative and $b$ has
infinite order in $K$, the same argument as above, applied directly in $K$, shows that
Eq.~\eqref{prop:rago-cardinal-lemma:companion-equ} has precisely $2^{|B|}$ solutions. Consequently, $2^{|A|}=2^{|B|}$, and hence $|A| = |B| = |f(A)|$, as desired.
\end{proof}

The relevance of Proposition~\ref{prop:rago-cardinal-lemma} to the broader
picture of the paper lies in the role it plays in establishing a key
property (Proposition \ref{prop:h-valuation-is-preserved}) of the $\mathbb N$-valued function introduced in the next
definition.

\begin{definition}
\label{def:h-adic-valuation}
Given a semigroup $H$ and an element $h \in H$, we refer to the function
\[
\mathcal P_\fin(H) \to \mathbb{N} \colon A \mapsto |\{a\in A: ha\in A\}|
\]
as the (\evid{left}) \evid{$h$-autocorrelation} on $\mathcal P_{\fin}(H)$, and denote it by $\nu_h$ whenever there is no risk of confusion.
\end{definition}

When $H$ is a group (usually abelian in the relevant
literature), this definition is closely related, via
Lemma~\ref{lem:formula-for-correlation} below, to the discrete covariogram in convex
geometry \cite[Sec.~6.1]{Bia-2023} and to difference representation
functions in additive combinatorics \cite[Sec.~6.3]{Nath-2009}. However, in both of these
settings, the set $A$ is fixed and the ``shift parameter'' $h$ is allowed to vary, whereas we fix $h$ and regard $\nu_h$ as a function on $\mathcal P_\fin(H)$. Moreover,
our definition makes sense for arbitrary, possibly non-commutative
semigroups, even when no interpretation in terms of difference sets is
available.

\begin{lemma}
\label{lem:formula-for-correlation}
Let $H$ be a monoid and fix $h \in H$. If $h$ is left-cancellative, then 
\[
\nu_h(A) = |A \cap hA|,
\qquad\text{for all }
A \in \mathcal P_\fin(H).
\]
\end{lemma}	

\begin{proof}
Set $X := \{a \in A: ha \in A\}$, and note that $\nu_h(A) = |X|$. Next, let $\phi$ be the function $
X \to H: a \mapsto ha$.
Since $h$ is left-cancellative, $\phi$ is injective. Moreover, it is clear that $\phi[X] \subseteq A \cap hA$, and it takes a moment to check that this inclusion is actually an equality. Indeed, if $x \in A \cap hA$, then $x = ha$ for some
$a \in A$, and $x \in A$ yields $a \in X$. Therefore, $\nu_h(A) = |X| = |A \cap hA|$, as claimed.
\end{proof}

\begin{proposition}
\label{prop:h-valuation-is-preserved}
Let $H$ and $K$ be cancellative monoids, and assume that $H$ contains an element of infinite order. If $g$ is the pullback of an isomorphism $f \colon \mathcal{P}_{\fin,1}(H) \to \mathcal{P}_{\fin,1}(K)$, then
\[
\nu_{h}(A) = \nu_{g(h)}(f(A)),
\qquad\text{for all }
h \in H \text{ and }
A \in \mathcal{P}_{\fin,1}(H).
\]
\end{proposition}

\begin{proof}
Fix $h\in H$ and $A\in \mathcal{P}_{\fin,1}(H)$. By left cancellativity, we have $|hA| = |A|$. Accordingly, it follows from the inclusion-exclusion principle and Lemma \ref{lem:formula-for-correlation} that
\[
|\{1_H, h\}A|=|A \cup hA| = |A| + |hA| - |A \cap hA| = 2\,|A|-\nu_{h}(A).
\]
In a similar way, we find that
\[
|\{1_K,g(h)\}f(A)| = 2\,|f(A)| - \nu_{g(h)} (f(A)).
\]

On the other hand, we know from Proposition \ref{prop:rago-cardinal-lemma} that $|X| = |f(X)|$ for every $X \in \mathcal P_{\fin,1}(H)$. Thus, using that $f$ is an isomorphism $\mathcal P_{\fin,1}(H) \to \mathcal P_{\fin,1}(K)$ and $g$ is its pullback gives
\[
\begin{split}
2\,|A|-\nu_{h}(A) 
& = |\{1_H,h\}A| = |f(\{1_H,h\}A)| = |f(\{1_H,h\})f(A)| \\
	&=|\{1_K,g(h)\}f(A)| = 2\,|f(A)|-\nu_{g(h)}(f(A)) = 2\,|A|-\nu_{g(h)}(f(A)).
\end{split}
\]
Clearly, this yields $\nu_{h}(A)=\nu_{h}(f(A))$ and completes the proof.
\end{proof}
	
We close the section with the following corollary, which is straightforward from Proposition \ref{prop:h-valuation-is-preserved}.

\begin{corollary}\label{nu-preserves}
If $H$ is a cancellative monoid and $f$ is an automorphism of $\mathcal P_{\fin,1}(H)$ with trivial pullback, then $
\nu_{h}(A)=\nu_{h}(f(A))$ for all $h \in H$ and $A \in \mathcal{P}_{\fin,1}(H)$.
\end{corollary}


\section{Focus on numerical monoids}
\label{sec:03}

Having established the necessary preliminaries in Section~\ref{sec:preliminaries}, we now turn our attention to numerical monoids. We shall use, in particular, the following result, which is borrowed from
\cite[Lemma~3.1]{Tri-Wen-2026(a)}.

\begin{lemma}
\label{lem:two-element-sets-are-fixed}
If $H$ is a numerical monoid and $f$ is an automorphism of $\mathcal P_{\fin,0}(H)$, then the pullback of $f$ is trivial. Equivalently, $
f(\{0, x\}) = \{0, x\}$ for every $x \in H$.
\end{lemma}

We begin with a generalization of \cite[Corollary 2.6(i)]{Tri-Yan2025(b)} from $\mathbb N$ to an arbitrary numerical monoid. The proof proceeds along similar lines, but is somewhat smoother in our setting owing to Lemma \ref{lem:two-element-sets-are-fixed}, which was not available when \cite{Tri-Yan2025(b)} was published.

\begin{lemma}
\label{lem:max-is-preserved}
If $H$ is a numerical monoid and $f$ is an automorphism of $\mathcal P_{\fin,0}(H)$, then 
\[
\max(A)=\max f(A),
\qquad\text{for all } A \in \mathcal P_{\fin,0}(H).
\]
\end{lemma}

\begin{proof}
Let $A \in \mathcal P_{\fin,0}(H)$. It follows from \cite[Lemma 2.2]{Tri-Yan-23(a)} that there exists an integer $n \ge 0$ such that 
\[
(n+1)A = nA + \{0, \max A\}.
\]
Applying $f$ to this latter equation and using Lemma yields
\[
(n+1) f(A) = n f(A) + f(\{0, \max A\}) = n f(A) + \{0, \max A\}.
\]
So, taking maxima and considering that 
\[
\max(kX) = k \max X,
\qquad\text{for all }
k \in \mathbb N
\text{ and }
X \in \mathcal P_\fin(\mathbb Z),
\]
we conclude that
\[
(n+1)\max f(A) = n \max f(A) + \max A,
\]
which shows that $\max f(A) = \max A$ and completes the proof.
\end{proof}

We continue with a generalization of items (ii) and (iii) of \cite[Lemma~2.11]{Tri-Yan2025(b)}. In the course of the proof, we will freely use a couple of \textit{auxiliary} results from \cite[Section~2]{Tri-Yan2025(b)} that are specific to $\mathbb N$. This should not be surprising, since the cases $H=\mathbb N$ and $H\ne\mathbb N$ are inherently different in the statement of Theorem~\ref{thm:main}, and this structural distinction must be reflected somewhere in the proofs.

First, note that, since $H$ is a cofinite subset of $\mathbb N$ (by the definition itself of a numerical monoid), there is a least integer $c \ge 0$, commonly called the \evid{conductor} of $H$, with the property that $\llb c, \infty \rrb \subseteq H$. Clearly, $c = 0$ if and only if $H = \mathbb N$, whereas $c \ge 2$ whenever $H \ne \mathbb N$.

\begin{lemma}
\label{lem:3-element-sets}
Let $H$ be a numerical monoid with conductor $c$ and $f$ be an automorphism of $\mathcal P_{\fin,0}(H)$. If $f$ does not fix the set $\{0, \alpha, \alpha+1\}$ for some integer $\alpha \ge c$, then 
\[
H = \mathbb N
\quad\text{and}\quad
f(\{0, \alpha, \alpha+1\}) = \{0, 1, \alpha+1\}.
\]
\end{lemma}

\begin{proof}
Set $E_\alpha := \{0, \alpha, \alpha+1\}$. If $\alpha = 0$ or $\alpha = 1$, then $c = 0$ and hence $H = \mathbb N$; since $\{0, 1, 2\} = 2 \{0, 1\}$, this would however imply by Lemma \ref{lem:two-element-sets-are-fixed} that $E_\alpha$ is a fixed point of $f$ (which is absurd). 

It is therefore clear that $\alpha \ge 2$ and $|E_\alpha| = 3$.
By Proposition \ref{prop:rago-cardinal-lemma} and Lemma \ref{lem:max-is-preserved}, it follows that 
\[
|f(E_\alpha)| = 3
\quad\text{and}\quad
\max f(E_\alpha) = \alpha+1. 
\]
Consequently, there exists $x \in H$ with $1 \le x < \alpha$ such that $E_\alpha \ne f(E_\alpha) = \{0, x, \alpha+1\}$; in particular, it holds that $x \ne \alpha$ as, otherwise, we would have $f(E_\alpha) = E_\alpha$.

Denote by $\nu_\alpha$ the $\alpha$-correlation on $\mathcal P_{\fin,0}(H)$, as introduced in Definition \ref{def:h-adic-valuation}.
Since the pullback of $f$ is trivial (Lemma \ref{lem:two-element-sets-are-fixed}), we gather from Corollary~\ref{nu-preserves} that 
\begin{equation}
\label{lem:sets-fixed-by-f:eq(1)}
\nu_\alpha(f(E_\alpha)) = \nu_\alpha(E_\alpha). 
\end{equation}
On the other hand, using that $\alpha \ge 2$ and hence $\alpha + 1 < 2\alpha$, we obtain from Lemma~\ref{lem:formula-for-correlation} that
\[ 
\nu_\alpha(E_\alpha) = 
\bigl| 
\{0,\alpha,\alpha+1\} \cap \{\alpha,2\alpha,2\alpha+1\} 
\bigr|
= 1.
\]

We are left to show that $1 \in H$ (note that $1 \in H$ if and only if $H = \mathbb N$). Suppose to the contrary that $x > 1$. Then $x < \alpha < \alpha+1 < \alpha+x < 2\alpha+1$, and thus
\[ 
\{0,x,\alpha+1\} \cap \{\alpha,\alpha+x,2\alpha+1\} = \emptyset. 
\] 
Another application of Lemma~\ref{lem:formula-for-correlation} then yields $\nu_\alpha(f(E_\alpha))=0$, in contradiction with Eq.~\eqref{lem:sets-fixed-by-f:eq(1)}. 
\end{proof}

\begin{lemma}
\label{lem:sets-fixed-by-f}
Let $H$ be a numerical monoid with conductor $c$, and let $f$ be an automorphism of $\mathcal P_{\fin,0}(H)$. Assume that $\{0,c+2,c+3\}$ is a fixed point of $f$. Then $f$ fixes the set
\[
\textstyle \{0\}\cup \left\llb a, na + \frac{1}{2}n(n+1) \right\rrb
\]
for all $a, n \in \mathbb N$ such that $a \ge c+2$ and $n \ge a+1$. 
\end{lemma}

\begin{proof}
Let $a$ be an integer with $a \ge c+2$, and let $n \ge a+1$. By \cite[Proposition 2.10]{Tri-Yan2025(b)}, we have
\[ A := \{0\} \cup \left\llb a, na + \tfrac{1}{2} n(n+1) \right \rrb = \sum_{i=0}^{n-1} \{0, a+i, a+i+1\},
\]
and hence
\[
f(A) = \sum_{i=0}^{n-1} f(\{0, a+i, a+i+1\}).
\]
It therefore suffices to prove that $f$ fixes the set $E_\alpha := \{0, \alpha, \alpha+1\}$ for every $\alpha \ge c+2$.

Assume to the contrary that $E_\alpha \ne f(E_\alpha)$ for some integer $\alpha \ge 2$. It is then immediate from Lemma \ref{lem:3-element-sets} that $H = \mathbb N$ and hence $c = 0$. Therefore, $f$ is an auto\-mor\-phism of $\mathcal P_{\fin,0}(\mathbb N)$ that fixes the set $\{0, 2, 3\}$. By \cite[Lemma 2.11(i)]{Tri-Yan2025(b)} applied to the inverse of $f$, it follows that $1 \in E_\alpha$, which is however absurd.
\end{proof}

The next lemma is the last ingredient in the proof of the main theorem of the paper (Theorem~\ref{thm:main}).

\begin{lemma}
\label{lem:fix-one-to-fix-them-all}
Let $H$ be a numerical monoid with conductor $c$, and let $f$ be an automorphism of $\mathcal P_{\fin,0}(H)$. If $f$ fixes the set $\{0, c+2, c+3\}$, then it is trivial.
\end{lemma}

\begin{proof}
Let $A \in \mathcal{P}_{\fin,0}(H)$. We need to show that $f(A) = A$. By Lemma \ref{lem:max-is-preserved}, we have
\begin{equation}
\label{lem:fix-one-to-fix-them-all:eq(1)}
m := \max A = \max f(A). 
\end{equation}
Choose a sufficiently large integer $a$, say
\begin{align}\label{eq:a-lower-bound}
a \geq \max (2m+1, m+c+2).
\end{align}

Since $a \ge c+2$, we are guaranteed by Lemma \ref{lem:sets-fixed-by-f} that there exists $n \geq a+m$ such that 
\begin{equation}
\label{lem:fix-one-to-fix-them-all:eq(2)}
f(B) = B,
\qquad \text{where } 
B := \{0\} \cup \llb a, n\rrb \in \mathcal P_{\fin,0}(H).
\end{equation}
Set $I := \llb a, m+n \rrb$. If $X \in \mathcal P_{\fin,0}(H)$ and $\max X = m$, then
\[
I = \llb a, n \rrb \cup \llb m+a, m+n \rrb = \{0, m\} + \llb a, n \rrb \subseteq X + \llb a, n \rrb \subseteq \llb 0, m \rrb + \llb a, n \rrb = I.
\]
Consequently, we conclude from Eq.~\eqref{lem:fix-one-to-fix-them-all:eq(1)} that
\begin{equation}
\label{eq:B+A-first}
A + B = (A + 0) \cup(A + \llb a, n\rrb)  = A \cup I.
\end{equation}

In a similar fashion, Eqs.~~\eqref{lem:fix-one-to-fix-them-all:eq(1)} and \eqref{lem:fix-one-to-fix-them-all:eq(2)} give
\begin{equation}\label{eq:B+f(A)-first}
f(A + B) = f(A) + B = f(A) \cup I.
\end{equation}

Next, fix $x \in \llb 0, m \rrb \cap H$, and set $h := a - x$. From Eqs.~\eqref{lem:fix-one-to-fix-them-all:eq(1)} and \eqref{eq:a-lower-bound}, we obtain
\begin{equation}
\label{equ:disjointness}
A \cap I = f(A) \cap I = \emptyset
\end{equation}
and
\begin{equation}
\label{equ:bound-on-h}
h \ge a - m \ge \max(m+1, c+2).
\end{equation}
In particular, $h \in H$. Accordingly, let $\nu_h$ be the $h$-autocorrelation on $\mathcal P_{\fin,0}(H)$, as per Definition \ref{def:h-adic-valuation}. 
By Lemma \ref{lem:two-element-sets-are-fixed}, the pullback of $f$ is trivial. Thus, Corollary~\ref{nu-preserves}, together with Eqs.~\eqref{eq:B+A-first} and \eqref{eq:B+f(A)-first}, implies
\begin{equation}\label{eq:nu-are-equal}
\nu_h(A \cup I) = \nu_h(f(A) \cup I).
\end{equation}

On the other hand, for each $y \in \mathbb N$, Eqs.~\eqref{lem:fix-one-to-fix-them-all:eq(1)} and \eqref{equ:bound-on-h} imply that $y+h \geq h > \max A$, and hence $y + \allowbreak h \notin A$. Therefore, since $A$ and $I$ are disjoint by Eq.~\eqref{equ:disjointness}, we find that
\begin{equation}
\label{equ:formula-for-v_h(A+B)}
\begin{split}
\nu_h(A \cup I) 
    & = {\bigl|\{y \in A \cup I: y+h \in A \cup I\}\bigl|} 
    \\
    & = \bigl|\{y \in A \cup I: y+h \in I\}\bigl|
    \\
    & = {\bigl|\{y \in A: y+h \in I\}\bigl|} 
    + {\bigl|\{y \in I: y+h \in I\}\bigl|} 
    \\
    & = {\bigl|\{y \in A: y+h \in I\}\bigl|} + {\bigl|\llb a,n+m-h\rrb\bigl|}.
\end{split}
\end{equation}

We are left to compute the first summand in the rightmost expression of the last displayed chain of equalities.
Pick $y \in A$. By the definition of $h$ (and the fact that $x \ge 0$), we have 
\[
y+h = a + (y - x) \le a + m \le n.
\]
That is, the upper bound required for $y+h$ to belong to $I$ is automatically satisfied. It follows that $y+ \allowbreak h \in \allowbreak I$ if and only if $a + y - x \ge a$, if and only if $y \ge x$. 
From Eq.~\eqref{equ:formula-for-v_h(A+B)}, it is then immediate that
\begin{equation}
\label{eq:nu_h(C+A)-third}
\nu_h(A \cup I) = {\bigl|A \cap \llb x, m \rrb \bigl|} + {\bigl|\llb a, n+m-h \rrb\bigl|}.
\end{equation}

Mutatis mutandis, the same argument shows that
\begin{equation}
\label{eq:nu_h(C+f(A))}
\nu_h(f(A) \cup I) = {\bigl|f(A) \cap \llb x, m\rrb \bigl|} + {\bigl|\llb a, n+m-h\rrb\bigl|}.
\end{equation}

So, stitching the pieces together, we are in a position to conclude from Eqs.~\eqref{eq:nu-are-equal}, \eqref{eq:nu_h(C+A)-third}, and \eqref{eq:nu_h(C+f(A))} that 
\begin{equation}
\label{eq:circled-equation}
\bigl|A \cap \llb x, m\rrb \bigl| = 
\bigl|f(A) \cap \llb x, m\rrb \bigl|,
\qquad 
\text{ for all } 
x \in \llb 0, m \rrb \cap H.
\end{equation}

Finally, suppose by way of contradiction that $A \ne f(A)$. There then exists a largest integer $x_0$ in the symmetric difference of $A$ and $f(A)$, and we may assume by symmetry that $x_0 \in A \setminus f(A)$. By Eq.~\eqref{lem:fix-one-to-fix-them-all:eq(1)} and the maximality of $x_0$, it follows that
\[
A \cap \llb x_0+1, m \rrb = f(A) \cap \llb x_0+1, m \rrb,
\]
and hence 
\[
\bigl| A \cap \llb x_0, m \rrb \bigr| = 1 + \bigl| f(A) \cap \llb x_0, m \rrb \bigr|.
\]
This is, however, in contradiction with Eq.~\eqref{eq:circled-equation} and completes the proof.
\end{proof}

\section{Proof of Theorem~\ref{thm:main}}
\label{sec:proof-of-main-thm}

Let $f$ be an automorphism of $\mathcal P_{\fin,0}(H)$, and let $c$ be the conductor of $H$. 
If $f$ fixes $\{0, c+2, c+3\}$, then Lemma \ref{lem:fix-one-to-fix-them-all} ensures that $f$ is the trivial automorphism. Otherwise, we have from Lemma \ref{lem:3-element-sets} that $H = \mathbb N$ and $f$ maps $\{0, 2, 3\}$ to $\{0, 1, 3\}$. Therefore, the function
\[
f^\prime \colon \mathcal P_{\fin,0}(\mathbb N) \to \mathcal P_{\fin,0}(\mathbb N) \colon X \mapsto \max X - f(X),
\]
fixes $\{0, 2, 3\}$. 
On the other hand, we are guaranteed by \cite[Lemma 2.7]{Tri-Yan2025(b)} that $f' \in \aut(\mathcal P_{\fin,0}(\mathbb N))$. Again by Lemma \ref{lem:3-element-sets}, it follows that $f'$ is the identity on $\mathcal P_{\fin,0}(\mathbb N)$, that is,
\[
\max X - f(X) = X, 
\qquad
\text{for every }
X \in \allowbreak \mathcal P_{\fin,0}(\mathbb N). 
\]
Thus, $f$ is the reversion map defined in Eq.~\eqref{equ:reversion-map}, and the proof of the theorem is complete.

\section*{Acknowledgments}

The authors were supported by the 111 Center under Grant No.~D26018 and by the Natural Science Foundation of Hebei Province under Grant No.~A2023205045.

\end{document}